\documentclass[11pt]{article}
\usepackage[dvips]{graphicx}
\usepackage{amssymb}
\usepackage{amsmath}
\usepackage{amsthm}
\usepackage{latexsym}
\usepackage[english]{babel}

\newtheorem{defi}{Definition}[section]
\newtheorem{teo}[defi]{Theorem}

\newtheorem{oss}[defi]{Remark}

\newtheorem{cor}[defi]{Corollary}

\begin{document}


\title{A symmetry result on submanifolds of space forms and applications}

\author{Ali Maalaoui$^{(1)}$ \& Vittorio Martino$^{(2)}$}
\addtocounter{footnote}{1}
\footnotetext{Department of Mathematics,
Rutgers University - Hill Center for the Mathematical Sciences
110 Frelinghuysen Rd., Piscataway 08854-8019 NJ, USA. E-mail address:
{\tt{maalaoui@math.rutgers.edu}}}
\addtocounter{footnote}{1}
\footnotetext{Dipartimento di Matematica, Universit\`a di Bologna,
piazza di Porta S.Donato 5, 40127 Bologna, Italy. E-mail address:
{\tt{martino@dm.unibo.it}}}
\date{}
\maketitle

{\noindent\bf Abstract} In this paper we prove a symmetry result on submanifolds of codimension one in a $n+1$-dimensional space form, related to the geodesic distance function and to the normal curvature of some fixed vector field. As applications we will prove sphere characterization type theorems for K\"{a}hler manifolds endowed with a toric group action.

\section{Introduction}
In this paper we will prove a symmetry result on submanifolds of codimension one (hypersurface type) in a $n+1$-dimensional space form, related to the geodesic distance function and to the normal curvature of some vector field. We recall that a space form is by definition a complete Riemannian manifold with constant sectional curvature $K$. Let then $V:=V^{n+1}$, $n\geq 1$ be a smooth complete manifold of dimension $n+1$ and let $g$ be a Riemannian metric on $V$ and $\nabla$ be the related Levi-Civita connection of $g$. Let us consider $M:=M^n$ a smooth, compact, orientable and with no boundary, embedded manifold on $V$, of dimension $n$. We consider on $M$ the metric induced by $g$ and the related induced Levi-Civita connection. Let us call $N$ the inner unit normal to $M$. The Second Fundamental Form $h$ on $M$ is the symmetric tensor defined on $TM$, the tangent bundle of $M$ , in the following way:
$$
h(X,Y)=g(\nabla_X  Y,N), \qquad \forall\; X,Y \in TM
$$
\begin{defi}\label{normalcurvaturvectorfield}
Let us consider a unit smooth vector field $X\in TM$. We will call the Normal Curvature of $M$ referred to the vector field $X$ the coefficient of the Second Fundamental Form related to the  vector field $X$:
$$\mathcal{C}^X:=h(X,X)=g(\nabla_X  X,N)$$
\end{defi}

\noindent
The distance function $d$ on $V$, related to $g$,  is defined as follows: let us consider any two points $p_0,p_1 \in V$, then
$$d(p_0,p_1)=\int_{t_0}^{t_1}\sqrt{g\big(\dot{r}(t),\dot{r}(t)\big)}\;dt$$
where
$$r:[t_0,t_1]\rightarrow V, \qquad r(t_0)=p_0, \quad r(t_1)=p_1$$
is a minimal geodesic in $V$ connecting $p_0$ and $p_1$. For $q\in V$, if $R>0$ is such that $exp_q$ (the exponential map at $q$) is a diffeomorphism on the ball $B(0,R)\subset T_q V$ then the geodesic ball $\mathcal{B}(q,R)$ of center $q$ and radius $R$ is the image set in $V$
$$\mathcal{B}(q,R)=exp_q(B(0,R))$$
moreover if the closed ball $\overline{B(0,R)}$ is also contained in an open set $U\subset T_q V$ where $exp_q$ is a diffeomorphism, then the geodesic sphere is $\partial\mathcal{B}(q,R)=exp_q(\partial B(0,R))$.\\
We can now state our result:
\begin{teo}\label{main}
Let $V$ be a $n+1$-dimensional Riemannian manifold with constant sectional curvature $K$ and let $M$ be a smooth, orientable, compact, connected, with no boundary, embedded manifold on $V$, of dimension $n$. Suppose there exist a point $q \notin M$ and a non-singular vector field $X \in TM$ such that
\begin{equation}\label{constantdistance}
X(d(q,p))=0, \qquad \forall\; p \in M
\end{equation}
Let $\mathcal{C}^X $ be the Normal Curvature of $M$ referred to the vector field $X$. We have two cases:
$$
\begin{array}{ll}
    i)  & \hbox{ $K\leq 0$. If $\mathcal{C}^X $ is constant on $M$, then $M$ is a geodesic sphere in $V$;} \\
    ii) & \hbox{ $K > 0$. If $M$ belong to $\mathcal{B}(q,\frac{\pi}{\sqrt{K}})$ and $\mathcal{C}^X $ is
           constant on $M$, then $M$}\\
     &    \hbox{  is a geodesic sphere in $V$.}
\end{array}
$$

\end{teo}
\begin{oss}\label{remarkmain}
We recall that fixed a point $q \in V$ then a conjugate point to $q$ is a point $p \in V$ such that there is no uniqueness for minimal geodesic connecting $q$ and $p$; in particular $exp_q$ fails to be a local diffeomorphism near $p$.  Then when $K>0$ we need to assume that $M$ is contained in the geodesic ball $\mathcal{B}(q,\frac{\pi}{\sqrt{K}})$ in order to avoid conjugate points. In fact by a classical result we have that if $V$ has positive constant sectional curvature $K$, then the first conjugate point along any geodesic starting from $q$ occurs at distance at least $\frac{\pi}{\sqrt{K}}$. On the other hand when $K\leq0$ then there are no conjugate points to any point of $V$.
\end{oss}

\noindent
We will give three applications of this result: the first one on almost symplectic manifolds, with the Normal Curvature referred to the hamiltonian vector field of the hypersurface $M$; the second one on K\"{a}hler manifolds endowed with a toric group action; finally we will specialize at the case of Reinhardt domains in $\mathbb{C}^{n+1}$ where we obtain as corollary the result in \cite{io2}: in that paper the second author, motivated by two recent works by Hounie and Lanconelli (\cite{HL}, \cite{HL1}), proved an Alexandrov type theorems for Reinhardt domains in $\mathbb{C}^{n+1}$ using the Characteristic Curvature rather than the Levi ones. Next we prove our result Theorem (\ref{main})
\begin{proof}
First of all, since the vector field $X$ is non-singular we can always normalize it such that $g(X,X)=1$.
For every $p \in M$ let us consider an integral curve $\gamma$ of any vector field $Y \in TM$ passing through $p$, namely: let $\varepsilon>0$, then $\gamma:(-\varepsilon,\varepsilon)\rightarrow M$ such that
\begin{equation}\label{integralcurve}
\gamma(0)=p,\quad \mbox{and} \quad \frac{d}{ds}\gamma(s)=Y_{\gamma(s)}, \quad \forall s \in (-\varepsilon,\varepsilon)
\end{equation}
Since there are no conjugate points to $q$ on $M$, then the exponential map $exp_q$ has no singularities on $M$. We consider the smooth family $r(s,t)$ of unique minimal geodesics connecting $q$ and $\gamma$, that is:
$$r(s,t):(-\varepsilon,\varepsilon)\times[0,1]\longrightarrow V$$
\begin{equation}\label{familygeodesics}
\left\{
\begin{array}{l}
\displaystyle \frac{D}{\partial t}\frac{\partial}{\partial t}r(s,t)=\nabla_{\dot{r}} \dot{r}=0\\
\\
r(s,0)=q,\qquad r(s,1)=\gamma(s)\\
\end{array}\right.
\end{equation}

\begin{figure}[h!]
   \centering
  {\includegraphics[width=5.0in,height=3.0in]{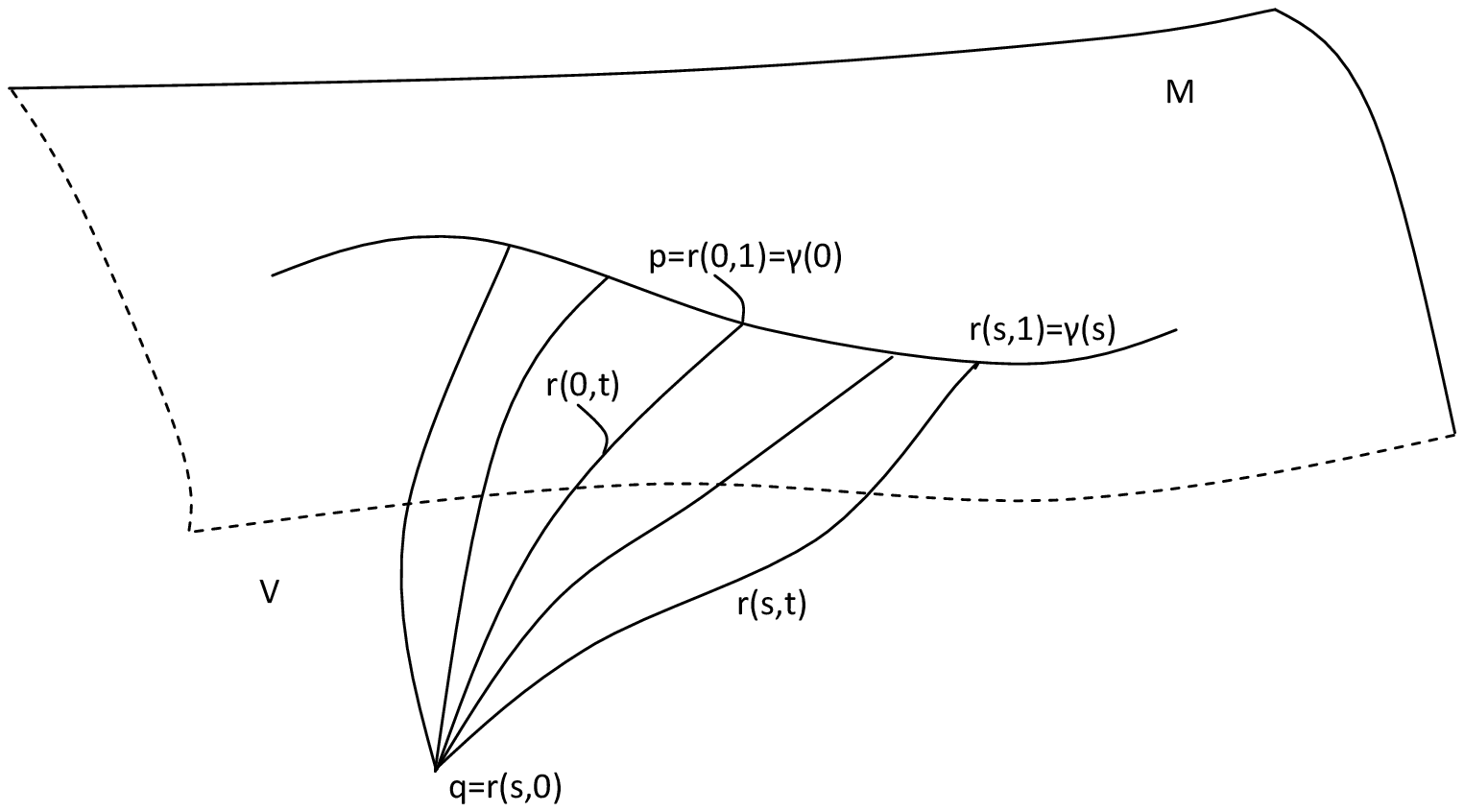}}
\end{figure}

\noindent
We will always denote by a dot $\cdot$ the derivative with respect to $t$. Moreover, by setting
$$\ell_{\gamma(s)}:=d(q,\gamma(s))$$
then we have parametrized the geodesics such that they have constant speed (with respect to t)
$$|\dot{r}(s,t)|=\ell_{\gamma(s)}, \qquad \forall\; t \in [0,1]$$
Let us define the following function defined on $M$
$$E:M\rightarrow \mathbb{R},\qquad E(p):=\frac{\Big( d(q,p)\Big)^2}{2}$$
Thus we have
$$Y(E(p))=d(q,p)Y\big( d(q,p) \big)$$
for every $Y\in  TM$, and in particular, as $q\notin M$ we have $d(q,p)>0$, then the critical points on $M$ of
$E$ are the same of $d(q,\cdot)$.
We are going to take the first variation of $E$ along any $Y\in TM$. In order to do that we consider the functional
$$\varphi:=E\circ\gamma:(-\varepsilon,\varepsilon)\rightarrow \mathbb{R}$$
$$\varphi(s)=E(\gamma(s))=\frac{\Big( d(q,\gamma(s))\Big)^2}{2}=\frac{1}{2}\int_{0}^{1} g\big(\dot{r}(s,t),\dot{r}(s,t)\big)\;dt$$
Therefore
$$Y(\varphi(s))=\frac{d}{ds}\varphi(s)=$$
$$=\frac{1}{2}\int_{0}^{1} \frac{d}{ds} g\big(\dot{r}(s,t),\dot{r}(s,t)\big)\;dt=
\int_{0}^{1} g\big(\frac{D}{ds}\dot{r}(s,t),\dot{r}(s,t)\big)\;dt$$
Now
$$g\big(\frac{D}{ds}\dot{r}(s,t),\dot{r}(s,t)\big)=
g\big(\frac{D}{ds}\frac{d}{dt}r(s,t),\dot{r}(s,t)\big)=
g\big(\frac{D}{dt}\frac{d}{ds}r(s,t),\dot{r}(s,t)\big)=$$
$$=\frac{d}{dt}g\big(\frac{d}{ds}r(s,t),\dot{r}(s,t)\big)-
g\big(\frac{d}{ds}r(s,t),\frac{D}{dt}\dot{r}(s,t)\big)=$$
and
$$\frac{D}{dt}\dot{r}(s,t)=\nabla_{\dot{r}}\dot{r} =0$$
therefore by (\ref{integralcurve}) and (\ref{familygeodesics}) we have
$$Y(\varphi(s))=\int_{0}^{1} \frac{d}{dt}g\big(\frac{d}{ds}r(s,t),\dot{r}(s,t)\big)\;dt=$$
$$=\Big[ g\big(\frac{d}{ds}r(s,t),\dot{r}(s,t)\big) \Big]_0^1=g\big(Y_{\gamma(s)},\dot{r}(s,1)\big)$$
Let us suppose now that $p_0\in M$ is a critical point for $E$, then we have $Y(\varphi(0))=0$ for any vector field $Y \in T_{p_0}M$ and consequently
$$\dot{r}(0,1)=-\ell_{\gamma(0)}N=-d(q,p_0)N$$
Moreover by the assumption (\ref{constantdistance}) we have for every $p \in M$ that
$$X(\varphi(s))=g\big(X_{\gamma(s)},\dot{r}(s,1)\big)=0, \qquad \forall\; s \in (-\varepsilon,\varepsilon)$$
Now we are going to take the second variation of $E$, twice along $X$, that is: let
$$
\gamma(0)=p,\quad \mbox{and} \quad \frac{d}{ds}\gamma(s)=X_{\gamma(s)}, \quad \forall s \in (-\varepsilon,\varepsilon)
$$
then we obtain
\begin{eqnarray}\label{secondvariation}
\nonumber   && 0=X^2(\varphi(s))=\frac{d}{ds}g\big(X_{\gamma(s)},\dot{r}(s,1)\big)=\\
\nonumber   && =g\big(\frac{D}{ds}X_{\gamma(s)},\dot{r}(s,1)\big)+g\big(X_{\gamma(s)},\frac{D}{ds}\dot{r}(s,1)\big)= \\
\nonumber   && =g\big(\nabla_{X_{\gamma(s)}} X_{\gamma(s)},\dot{r}(s,1)\big)+g\big(X_{\gamma(s)},\frac{D}{ds}\dot{r}(s,1)\big)
\end{eqnarray}
Now if we evaluate the previous expression at a critical value $p_0=\gamma(0)$ of $E$, we get:
\begin{eqnarray}\label{secondvariationcriticalpoint}
\nonumber    && 0=g\big(\nabla_{X} X,\dot{r}(0,1)\big)_{p_0}+g\big(X_{p_0},\frac{D}{ds}\dot{r}(0,1)\big)=\\
\nonumber    && =-d(q,p_0)g\big(\nabla_{X} X,N\big)_{p_0}+g\big(X_{p_0},\frac{D}{ds}\dot{r}(0,1)\big)=\\
             && =-d(q,p_0)\mathcal{C}^X_{p_0}+g\big(X_{p_0},\frac{D}{ds}\dot{r}(0,1)\big)
\end{eqnarray}
We need to compute
\begin{equation}\label{derivatajacobi}
\frac{D}{ds}\dot{r}(0,1)=\frac{D}{ds}\dot{r}(s,t)\Bigg|_{s=0,t=1}=
\frac{D}{ds}\frac{d}{dt}r(s,t)\Bigg|_{s=0,t=1}=\frac{D}{dt}\frac{d}{ds}r(s,t)\Bigg|_{s=0,t=1}
\end{equation}
Now we explicitly note that
$$J(t):=\frac{d}{ds}r(s,t)\Bigg|_{s=0}$$
is a Jacobi field along the geodesic $r(0,\cdot)$, since it is a variation field through geodesics of the geodesic $r(0,\cdot)$. In particular $J$ is normal along $r(0,\cdot)$, in fact:
$$\frac{d}{dt} g\big(J(t),\dot{r}(0,t)\big)=\frac{d}{dt} g\big(\frac{d}{ds}r(s,t),\frac{d}{dt}r(s,t) \big)\Bigg|_{s=0}=$$
$$=g\big(\frac{D}{dt} \frac{d}{ds}r(s,t),\frac{d}{dt}r(s,t) \big)\Bigg|_{s=0}+
g\big(\frac{d}{ds}r(s,t),\frac{D}{dt}\frac{d}{dt}r(s,t) \big)\Bigg|_{s=0}=$$
$$=g\big(\frac{D}{ds}\frac{d}{dt}r(s,t),\frac{d}{dt}r(s,t) \big)\Bigg|_{s=0}=
\frac{1}{2}\frac{d}{ds}g\big(\frac{d}{dt}r(s,t),\frac{d}{dt}r(s,t) \big)\Bigg|_{s=0}=$$
$$=\frac{1}{2}\frac{d}{ds}|\dot{r}(s,t)|^2\Bigg|_{s=0}=\frac{1}{2}\frac{d}{ds}\big(\ell_{\gamma(s)}\big)^2\Bigg|_{s=0}=0$$
Therefore the function $g\big(J(t),\dot{r}(0,t)\big)$ is constant along $t$, and for $t=1$ we have:
$$g\big(J(1),\dot{r}(0,1)\big)=-d(q,p_0)g\big(X_{p_0},N_{p_0}\big)=0$$
Now we will use the hypothesis on the constant sectional curvature in order to write explicit formulas for the normal Jacobi fields. First we recall that if $K$ is the constant sectional curvature of $V$, then the Riemann curvature endomorphism $R$ satisfies the following identity:
\begin{equation}\label{riemannendomorphism}
R(X,Y)Z=K\Big(g\big(Y,Z\big)X-g\big(X,Z \big)Y \Big), \qquad \forall X,Y,Z \in TM
\end{equation}
Thus, as $J$ satisfies the Jacobi equation, we have
$$0=\frac{D^2}{dt^2}J(t)+R(J,\dot{r}(0,t))\dot{r}(0,t)=$$
$$=\frac{D^2}{dt^2}J(t)+K\Big(g\big(\dot{r}(0,t),\dot{r}(0,t)\big)J(t)-g\big(J(t),\dot{r}(0,t) \big)\dot{r}(0,t) \Big)=$$
$$=\frac{D^2}{dt^2}J(t)+|\dot{r}(0,t)|^2 KJ(t)=\frac{D^2}{dt^2}J(t)+\ell^2_{p_0} KJ(t)$$
where $\ell_{p_0}=d(q,p_0)$. Now by choosing any parallel vector field $Z$ along the geodesic $r(0,\cdot)$, namely
$$\frac{D}{dt}Z(t)=\nabla_{\dot{r}(0,t)}Z_{r(0,t)}=0$$
we can write $J(t)=u(t)Z(t)$, provided the scalar function $u$ satisfies the second order differential equation:
$$\ddot{u}(t)+\ell^2_{p_0} Ku(t)=0$$
In particular since we have $J(0)=0$ we require $u(0)=0$ and then we can choose $Z$ such that $Z(1)=X_{p_0}$ so that $u(1)=1$.
Finally, by setting $a=d(q,p_0)\sqrt{K}$ (for $K\neq 0$), we have the explicit formulas:
\begin{equation}\label{normaljacobifields}
u(t)=\left\{
\begin{array}{lc}
  \displaystyle\frac{1}{\sin(a)}\sin(at),   & K>0 \\
  \\
  t,                                        & K=0 \\
  \\
  \displaystyle\frac{1}{\sinh(a)}\sinh(at),  & K<0
\end{array}
\right.
\end{equation}
We want to note that if $K>0$ then by our hypothesis for every $p\in M$ it holds
$$0<d(q,p)<\frac{\pi}{\sqrt{K}} \Longrightarrow 0<a<\pi$$
Now by (\ref{derivatajacobi}) we have
$$\frac{D}{ds}\dot{r}(0,1)=\frac{D}{dt}\frac{d}{ds}r(s,t)\Bigg|_{s=0,t=1}=
\frac{D}{dt}J(t)\Bigg|_{t=1}=\dot{u}(1)Z(1)=\dot{u}(1)X_{p_0}$$
with
\begin{equation}\label{derivatau1}
\dot{u}(1)=\left\{
\begin{array}{lc}
  \displaystyle a \cot(a) ,   & K>0 \\
  1,                          & K=0 \\
  \displaystyle a \coth(a),   & K<0
\end{array}
\right.
\end{equation}
Therefore (\ref{secondvariationcriticalpoint}) rewrites as
$$0=-d(q,p_0)\mathcal{C}^X_{p_0}+g\big(X_{p_0},\dot{u}(1)X_{p_0}\big)=
-d(q,p_0)\mathcal{C}^X_{p_0}+\dot{u}(1)$$
Since $X$ is everywhere defined and non-zero, and $M$ is compact, then $E$ admits maximum and minimum which are critical point for $E$. By the hypothesis on the Normal Curvature, $\mathcal{C}^X_{p}=\mathcal{C}^X =const.$ on $M$, we have that for all the critical points $p_0$ of $E$, in particular on the maximum and on the minimum, it holds
\begin{equation}\label{raggiodistanza}
\left\{
\begin{array}{lc}
  \displaystyle d(q,p_0)=\frac{1}{\sqrt{K}} \arctan\Big(\frac{\sqrt{K}}{\mathcal{C}^X}\Big) ,          & K>0 \\
  \\
  d(q,p_0)=\displaystyle\frac{1}{\mathcal{C}^X},                                                       & K=0 \\
  \\
  \displaystyle d(q,p_0)=\frac{1}{\sqrt{K}} \textrm{arctanh}\Big(\frac{\sqrt{K}}{\mathcal{C}^X}\Big) , & K<0
\end{array}
\right.
\end{equation}
$$$$
Therefore the distance function $d(q,\cdot)$ is constant on $M$ and $M$ is a geodesic sphere.
\end{proof}

\begin{oss}
From the expression of $X^2(E(\cdot))$ evaluated at a maximum point $p_0$ of the distance function we have in particular
$$X^2(E(p_0))=-d(q,p_0)\mathcal{C}^X_{p_0}+\dot{u}(1) \leq 0$$
and since $d(q,p_0)$ is strictly positive we have that the constant prescribed for $\mathcal{C}^X$ in the hypothesis  necessarily satisfies
$$\mathcal{C}^X\geq \frac{\dot{u}(1)}{d(q,p_0)}$$
\end{oss}

\section{Some applications}
We are going to apply the Theorem (\ref{main}) to some manifolds with additional structures. First we treat the case of an almost symplectic manifold with a general symplectic group action, then we will specialize to the case of K\"{a}hler manifold with a toric group action; finally we consider the particular case of the Reinhardt domains in $\mathbb{C}^{n+1}$.\\
Let $V:=V^{2(n+1)}$ be a smooth differentiable manifold of dimension $2(n+1)$. We recall that an \emph{almost symplectic structure} on $V$ is a $2$-form $\omega$ everywhere non degenerate on $V$; $\omega$ is said symplectic if it is close. Moreover an \emph{almost complex structure} on $V$ is a map $J$ such that for every $p\in M$ $J_p$ is a smooth endomorphism on $T_p M$ with  $J_p^2=-1$; $J$ is said complex if it is integrable. For every fixed almost symplectic structure $\omega$ there exists and almost complex structure $J$ on $V$ and a compatible metric $g$ on $V$, that means that it holds
\begin{equation}\label{compatible}
\omega(X,Y)=g(X,JY)
\end{equation}
for every pair of vector fields $X,Y \in TV$.
Let us consider then any compatible triple $(\omega,J,g)$ on $V$ and a smooth embedded manifold $M$ on $V$, of codimension $1$. $M$ can always be seen (at least locally) as the (smooth) level set of some function $H$, namely,
$$H:V\rightarrow \mathbb{R}, \qquad M:=\{ H=0 \}, \qquad \nabla H\neq 0 \; on \; M$$
where $\nabla$ denotes the \emph{gradient} with respect to $g$.
Then it is always possible to define the hamiltonian (non-singular) vector field $X^H \in TM$ related to the hamiltonian function $H$ in the following way:
$$\omega(X^H,\cdot)=-dH(\cdot)$$
or equivalently, since $dH(\cdot)=g(\nabla H, \cdot)$
$$X^H=J \nabla H$$
We need to recall some known facts in order to state our next results. First we recall that a symplectic transformation on $V$ is a map $u:V\rightarrow V$ that satisfies $u^* \omega=\omega$.
A symplectic Lie group action $A$ on $V$ is a group action such that $a$ is a symplectic map on $V$, for every $a\in A$.
We have that if the symplectic Lie group action $A$ on $V$ is compact, then there always exists an almost complex structure $J$ on $V$ such that with respect the associated compatible metric $g$ is invariant under the action of $A$ (see for instance \cite{audin}): we will call this invariant metric $g_A$.
We can state now the following corollary of the Theorem (\ref{main}):
\begin{cor}
Let $(V,\omega)$ an almost symplectic manifold of dimension $2(n+1)$ and let $M$ be a smooth, orientable, compact, connected, with no boundary, embedded manifold on $V$, of codimension $1$. Let $A$ a symplectic compact Lie group action on $V$ with a fixed point $q$, such that $M$ is stable under the action of $A$. Suppose $V$ has non-positive constant sectional curvature with respect to the invariant metric $g_A$.\\
If the Normal Curvature related to the hamiltonian vector field $\mathcal{C}^{X^H}$ is constant on $M$, then $M$ is a geodesic sphere with respect to the metric $g_A$.
\end{cor}
\begin{oss}
For the sake of simplicity we stated the corollary only in the case of non-positive constant sectional curvature: with the additional hypothesis as in Theorem (\ref{main}) (see Remark (\ref{remarkmain})) one can handle the case of positive constant sectional curvature as well.
\end{oss}
\begin{proof}
First of all, since $A$ is a symplectic group action then $X^H$ is one of the vector fields generated by $A$; in particular since $M$ is smooth then $X^H$ never vanishes on $M$. Moreover the fixed point $q$ does not belongs to $M$: in fact if $q\in M$ would mean that all the vector fields generated by the action $A$ vanish at $q$ (by the very definition of fixed point, the orbit of $q$ under the action $A$ is the point $q$), and we know that the hamiltonian vector field $X^H$ (generated by $A$) never vanishes on $M$.  Then by considering the distance $d_A$ induced by the invariant metric $g_A$ on $V$, we have by direct computation that
\begin{equation}\label{almostsymplecticdistanceconstant}
X^H(d_A(q,p))=0, \qquad \forall\; p \in M
\end{equation}
where $q\notin M$ is the fixed point of the action $A$. Then Theorem (\ref{main}) holds.
\end{proof}

\noindent
Now we are going to consider the case of K\"{a}hler manifolds. Let $V:=V^{2(n+1)}$ be a smooth differentiable manifold of dimension $2(n+1)$. $V$ is said to be a K\"{a}hler manifold if there exists a symplectic structure $\omega$, a complex structure $J$ and a Riemannian metric $g$ such that they are compatible in the sense of (\ref{compatible}).
We recall that a $2(n+1)$-dimensional symplectic toric manifold is a compact connected
symplectic manifold $(V^{2(n+1)},\omega)$ equipped with an effective hamiltonian action $A$ of an
$n+1$-torus $\mathbb{T}^{(n+1)}$ and with a corresponding moment map. We will refer to $V$ as a K\"{a}hler toric manifold if in addition the toric action $A$ is holomorphic. We have the following result:
\begin{cor}
Let $V$ be a symplectic toric K\"{a}hler manifold of dimension $2(n+1)$ with non-positive constant sectional curvature and let $M$ be a smooth, orientable, compact, connected, with no boundary, real hypersurface on $V$, stable under the toric group action $A$. \\
If the Normal Curvature related to the hamiltonian vector field $\mathcal{C}^{X^H}$ is constant on $M$, then $M$ is a geodesic sphere.
\end{cor}
\begin{proof}
As in the previous proof, since $A$ is a symplectic group action then $X^H$ is one of the vector fields generated by $A$; moreover since $M$ is smooth then $X^H$ never vanishes on $M$. Moreover by the compactness of $V$ and by the results of Atiyah \cite{atiyah}, Guillemin-Sternberg \cite{guilster} and Bredon \cite{bredon} we have that there exist at least $n+2$ fixed points for the toric group action $A$ and by using the same argument as in the previous proof, we know that none of them is on $M$: let us choose one and let us call this fixed point $q\notin M$. In addition compatible metric $g$ on this K\"{a}hler manifolds is invariant under the group action $A$: in fact the group action $A$ is holomorphic and then it commutes with the compatible complex structure $J$. As consequence we have that then the following condition is satisfied:
\begin{equation}\label{kahlerdistanceconstant}
X^H(d_g(q,p))=0, \qquad \forall\; p \in M
\end{equation}
where $d_g$ is the distance function induced by the invariant metric $g$. Then Theorem (\ref{main}) holds.
\end{proof}

\noindent
Now we will consider $M$ as the smooth boundary of a Reinhardt domain in $\mathbb{C}^{n+1}$. A Reinhardt domain $\Omega$ (with center at the origin) is by definition an open subset of $\mathbb{C}^{n+1}$ such that
\begin{equation}\label{reinhardtdefinition}
\mbox{if} \quad (z_1,\ldots,z_{n+1})\in \Omega \quad \mbox{then} \quad (e^{i\theta_1} z_1,\ldots, e^{i \theta_{n+1}} z_{n+1} )\in \Omega
\end{equation}
for all the real numbers $\theta_1,\ldots,\theta_{n+1}$.
These domains naturally arise in the theory of several complex variables as the logarithmically convex Reinhardt domains are the domains of convergence of power series (see for instance \cite{hormandercomplex}, \cite{jarpfl}). The smooth boundary $M:=\partial\Omega$ is then a smooth real hypersurface in $\mathbb{C}^{n+1}$ and thus a CR-manifold of CR-codimension equal to one, with the standard CR structure induced by the holomorphic structure of $\mathbb{C}^{n+1}$. Thus for every $p\in M$ the tangent space $T_p M$ splits in two subspaces: the $2n-$dimensional horizontal subspace $H_p M$, the largest subspace in $T_p M$ invariant under the action of the standard complex structure $J$ of $\mathbb{C}^{n+1}$ and the vertical one-dimensional subspace generated by the characteristic direction $T_p :=J  N_p$, where $N_p$ is the unit normal at $p$. Moreover, if $g$ is the standard metric on $\mathbb{C}^{n+1}$, then it holds
$$T_p M=H_p M\oplus\mathbb{R}T_p$$
and the sum is $g$-orthogonal.\\
Let $\nabla$ be the Levi-Civita connection for $g$ and let us consider the complexified horizontal space
$$H^{\mathbb{C}} M :=\{Z= X-iJ \cdot X: X\in HM \}$$
The Levi Form $l$ is then the sesquilinear and hermitian operator on $H^{\mathbb{C}}M$ defined in the following way:
$\forall Z_1,Z_2\in H^{\mathbb{C}} M$
\begin{equation}\label{leviform}
l(Z_1,Z_2)= \widetilde{g}({\widetilde{\nabla}}_{Z_1} \bar Z_2, N)
\end{equation}
One then compares the Levi Form with the Second Fundamental Form $h$ of $M$ by using the identity \cite{bog}
$$
l(Z,Z)=h(X,X)+h(J(X),J(X)),\quad \forall X\in HM
$$
\begin{defi}\label{linea}
We will call $\mathcal{C}^T=h(T,T)=g({\nabla}_{T} T, N)$ the Characteristic Curvature of $M$.
\end{defi}

\noindent
Thus, a direct calculation leads to the relation between the classical Mean Curvature $H$, the Levi-Mean Curvature $L$ and the Characteristic Curvature $\mathcal{C}^T$ of $M$:
\begin{equation}\label{meanlevimean}
H=\frac{1}{2n+1}(2nL+\mathcal{C}^T)
\end{equation}
Following a couple of papers by Hounie and Lanconelli (\cite{HL}, \cite{HL1}) in which they prove Alexandrov type theorems for Reinhardt domains in $\mathbb{C}^{n+1}$ using the Levi Mean Curvature, the second author in \cite{io2} proved a similar symmetry result for Reinhardt domains in $\mathbb{C}^{n+1}$ starting from the Characteristic Curvature rather than the Levi ones:
\begin{teo}\label{teoreinhardt}
Let $M:=\partial\Omega$ be the smooth boundary of a bounded Reinhardt domain $\Omega$ in $\mathbb{C}^{n+1}$.
If the characteristic curvature $\mathcal{C}^T$ is constant then $M$ is a sphere of radius equal to $1/\mathcal{C}^T$.
\end{teo}
\noindent
Here we show that this result is a corollary of our main Theorem (\ref{main})
\begin{proof}
We can think of $\mathbb{C}^{n+1}$  as a K\"{a}hler manifold with the standard compatible symplectic, complex and metric structures and with sectional curvature identically zero. We recall now that for every hypersurface $M$ in $\mathbb{C}^{n+1}$, with $f$ as defining function, the characteristic direction $T$ of $M$ is exactly the (normalized) hamiltonian vector field for the hamiltonian function $f$. Moreover by the very definition of Reinhardt domain  (\ref{reinhardtdefinition}) we recognize that there exist an explicit toric group action $A$ on $\mathbb{C}^{n+1}$ such that $M$ is stable under $A$. Since $\mathbb{C}^{n+1}$ is non compact we note that we do not have a symplectic toric K\"{a}hler manifold, but in this particular situation we have that the origin is a fixed point for $A$ and it does not belong to $M$. Then Theorem (\ref{main}) holds.
\end{proof}


\addcontentsline{toc}{section}{Riferimenti Bibliografici}


\end{document}